\documentclass[12pt,reqno]{amsart}
\usepackage{dsfont}
\usepackage{tikz}
\usepackage{amssymb}
\usepackage{amsmath}
\usepackage{enumerate}

\usepackage{mathabx,epsfig}

\usepackage{changepage}

\newtheorem{fact}{Fact}[section]
\newtheorem{lemma}[fact]{Lemma}
\newtheorem{theorem}[fact]{Theorem}

\newtheorem{rremark}[fact]{Remark}
\newenvironment{remark}{\begin{rremark} \rm}{\end{rremark}}

\DeclareMathOperator\Id{Id}

\author{Samantha Moore}
\address{Department of Mathematics, University of North Carolina at Chapel Hill, USA}
\email{scasya@live.unc.edu}

\usepackage{lipsum}

\newcommand\blfootnote[1]{%
  \begingroup
  \renewcommand\thefootnote{}\footnote{#1}%
  \addtocounter{footnote}{-1}%
  \endgroup
}

\title{A Combinatorial Formula for the Bigraded Betti Numbers}

\begin{document}
\blfootnote{2020 \textit{Mathematics Subject Classification.} 55N31, 13C05.}
\blfootnote{This material is based upon work supported
by the National Science Foundation Graduate Research Fellowship under Grant No. 1650116.}

\begin{abstract}
It has been shown that $1$-parameter persistence modules have a very simple classification, namely there is a discrete invariant called a barcode that completely characterizes $1$-parameter persistence modules up to isomorphism. In contrast, Carlsson and Zomorodian showed that $n$-parameter persistence modules have no such ``nice" classification when $n>1$; every discrete invariant is incomplete. Despite their incompleteness, discrete invariants can still provide insight into the properties of multiparameter persistence modules. A well-studied discrete invariant for $2$-parameter persistence modules is the bigraded Betti numbers. Through commutative algebra techniques, it is known that the bigraded Betti numbers of a $2$-parameter persistence module $M$ can be recovered from the barcodes of certain zigzag modules within $M$ via a simple combinatorial formula. We present an alternate proof of this formula that relies only on basic linear algebra.
\end{abstract}

\maketitle


\section{Background}

\subsection{Multiparameter Persistence Modules}

Let $\{e_i\}_i$ denote the standard orthonormal frame of $\mathbb{N}^n$ and fix some field $\mathbb{F}$. An $n$\textbf{-parameter persistence module} $M$ is defined by assigning an $\mathbb{F}$-vector space $M_\alpha$ to each vertex $\alpha$ of the $\mathbb{N}^n$ lattice and a homomorphism $^M\phi_{\alpha}^{\alpha+e_i}:M_\alpha\rightarrow M_{\alpha+e_i}$ to each edge $\alpha\rightarrow \alpha +e_i$ of the lattice such that the resulting diagram commutes. Such modules are often build from data sets in order to study the structure of the date (see, for example, [CZ]). An example of a $2$-parameter persistence module is illustrated in Fig. \ref{fig:FirstExample}. The information $\{\dim(M_\alpha)\}_{\alpha\in\mathbb{N}^n}$ is called the \textbf{dimension vector} of $M$. 

\begin{figure}[h]
    \centering
    \includegraphics[scale=.5]{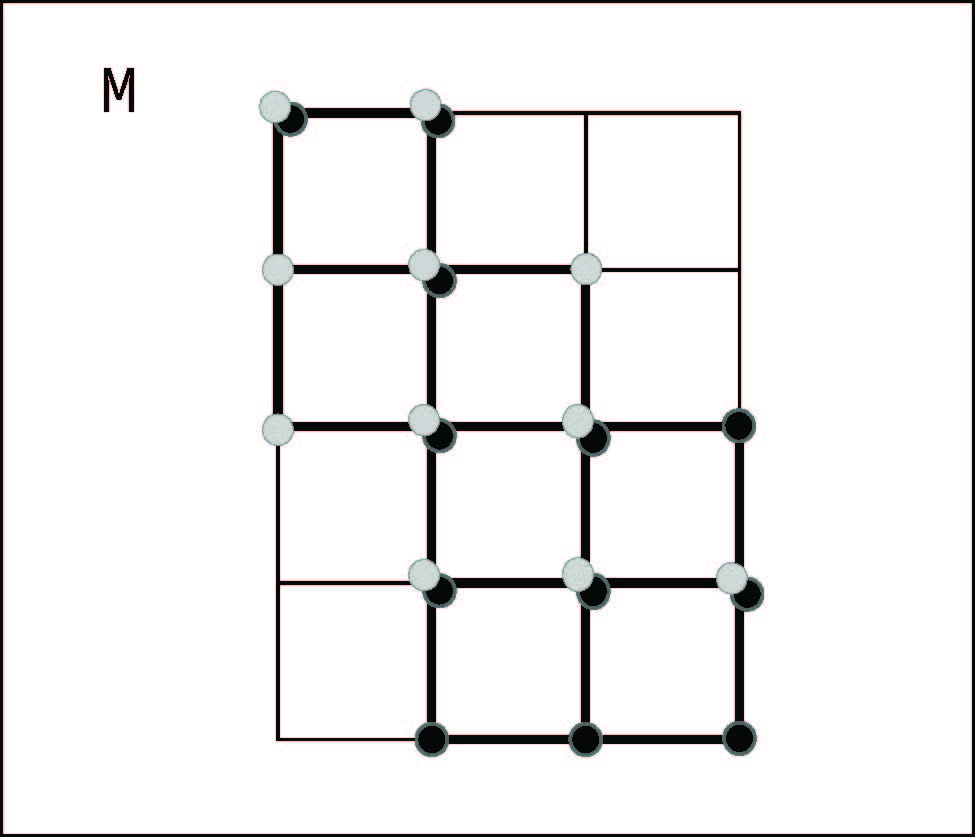}
    \vspace{.2cm}
    \caption{This is an example of a $2$-parameter persistence module $M$. Each colored dot at $\alpha\in\mathbb{N}^2$ represents a basis element of $M_\alpha$. Thick edges represent the identity maps on these colored components, meaning that black basis vectors map to black basis vectors, and similarly for gray basis vectors. Thinner edges represent the zero map.}
    \label{fig:FirstExample}
\end{figure}

We place a partial ordering  on $\mathbb{N}^n$ by defining $\alpha\leq\beta$ if the $i^{th}$ coordinates of such satisfy $\alpha_i\leq \beta_i$ for all $i$. Notice that the commutativity of the $^M\phi_{\alpha}^{\alpha+e_i}$ maps results in a single well-defined linear map from $M_\alpha$ to $M_\beta$ whenever $\alpha\leq\beta$. Denote this linear map by $^M\phi_{\alpha}^{\beta}: M_\alpha\rightarrow M_\beta$. For example, in Fig. \ref{fig:FirstExample} the map from $M_{(1,0)}$ to $M_{(3,2)}$ is $^M\phi_{(1,0)}^{(3,2)}=\Id_{\mathbb{F}}$.

A \textbf{homomorphism} $\mu$ between $n$-parameter persistence modules $M$ and $N$ is a collection of linear maps $\mu_\alpha: M_\alpha\rightarrow N_\alpha$ which commute with the $^N\phi$ and $^M\phi$ maps. That is, $$^N \phi_{\alpha}^{\beta}\circ \mu_\alpha=\mu_\beta\circ\ \! ^M\phi_{\alpha}^{ \beta}$$ for all $\alpha\leq\beta\in\mathbb{N}^n$. The notion of an isomorphism between $n$-parameter persistence modules is thus defined, leading naturally to the question of classification of multiparameter persistence modules up to isomorphism, which we explore in later sections.

\subsection{Correspondence to $\mathbb{N}^n$-graded Modules over $\mathbb{F}[x_1,...,x_n]$}
We restrict our attention to \textbf{finitely generated} $n$\textbf{-parameter persistence modules}, meaning that we require each $M_\alpha$ to be finite dimensional, and for each sequence $\alpha_1\leq\alpha_2\leq\alpha_3\leq\cdot\cdot\cdot$, there must exist an index $k$ such that $^M\phi_{\alpha_i}^{\alpha_{i+1}}$ is an isomorphism whenever $i\geq k$ [CK]. For example, the module in Fig. \ref{fig:FirstExample} is finitely generated. The category of finitely generated $n$-parameter persistence modules is equivalent to the category of finitely generated $\mathbb{N}^n$-graded modules over $\mathbb{F}[x_1,...,x_n]$. The bijection between the objects of these categories is as follows: Each such persistence module $M$ is mapped to the $\mathbb{N}^n$-graded module $M'$ over  $\mathbb{F}[x_1,...,x_n]$ with grading $M'=\bigoplus\limits_{\alpha\in\mathbb{N}^n}M_\alpha,$ and whose $\mathbb{F}[x_1,... ,x_n]$ action is given by $$x_i\cdot v:=\ \!^M\!\phi_{\alpha}^{\alpha+e_i}(v)$$ for all $v\in M_\alpha$ [CZ]. 

From this equivalence of categories, definitions related to  $\mathbb{N}^n$-graded modules over $\mathbb{F}[x_1,...,x_n]$ can be interpreted as definitions for $n$-parameter persistence modules. For example, the \textbf{direct sum} of two $n$-parameter persistence modules $M$ and $N$ is the $n$-parameter persistence module $P=M\bigoplus N$ with vector spaces $$P_\alpha:= M_\alpha \bigoplus N_\alpha$$ and maps $$^P\phi_\alpha^\beta:=\ \! ^M\phi_\alpha^\beta \bigoplus\ \!\! ^N\!\phi_\alpha^\beta.$$ A multiparameter persistence module $M$ is \textbf{indecomposable} if whenever $M=M_1\bigoplus M_2$, either $M_1=0$ or $M_2=0.$ 

\begin{theorem}[Krull-Schmidt-Remak] If $M$ is a $k$-parameter persistence module then $M\cong \bigoplus\limits_i M_i$ where each $M_i$ is indecomposable. Furthermore, this indecomposable decomposition of $M$ is unique up to ordering and isomorphism of the summands.
\end{theorem}

It is also natural to define \textbf{free} $n$\textbf{- parameter persistence modules} to be those whose image under the equivalence are free $\mathbb{N}^n$- graded $\mathbb{F}[x_1,...,x_n]$ modules. In particular, let $\mathbb{F}_n(\alpha)$ denote the $n$-parameter persistence module with vector spaces $$\displaystyle (\mathbb{F}_n(\alpha))_\beta= \begin{cases}
\mathbb{F} &  \alpha\leq \beta \in\mathbb{N}^n\\
 0 & \text{otherwise} \\ 
 \end{cases}$$ and maps $$\displaystyle ^{\mathbb{F}_n(\alpha)}\phi_{\beta}^\delta= \begin{cases}
Id_\mathbb{F} & \alpha\leq \beta\leq \delta \\ 
 0 & \text{otherwise.} \\ 
 \end{cases}$$
Free multiparameter persistence modules are the modules of the form $F(S):= \bigoplus\limits_{\alpha\in S}\mathbb{F}_n(\alpha)$, where S is a multiset.


\subsection{Classification of $1$-parameter persistence modules}
 In the one parameter case, persistence modules have a very simple classification; there is a discrete invariant (examined below) that completely characterizes finitely generated $1$-parameter persistence modules up to isomorphism [ZC]. This will be in direct contrast with the $n$-parameter case when $n>1$, whose classification is much more complicated, as we will see in the next section.
 

A \textbf{zigzag persistence module} $M$ is a generalization of a $1$-parameter persistence module [CdS]. Namely, let $[\alpha,\beta]$ be an interval in $\mathbb{N}$. To each $\gamma\in [\alpha,\beta]$ we associate a vector space $M_\gamma$ as well as a linear map $^M\phi_\gamma$ of either the form $^M\phi_\gamma=\!^M\phi_{\gamma}^{\gamma+1}:M_\gamma\rightarrow M_{\gamma+1}$ or the form $^M\phi_\gamma=\!^M\phi_{\gamma+1}^{\gamma}:M_{\gamma+1}\rightarrow M_{\gamma}.$ The Krull-Schmidt-Remak theorem also applies to zigzag persistence modules.

\begin{theorem} [Gabriel]
A zigzag persistence module $M$ is indecomposable  if and only if there is an interval $[\alpha,\beta]$ in $\mathbb{N}$ such that 
\begin{enumerate}
\item $M_\delta=\mathbb{F}$ for all $\delta\in[\alpha,\beta]$,
\item  $^M\! \phi_\delta=Id_\mathbb{F}$ for all $\alpha\leq\delta\leq\beta-1$, and 
\item the vector spaces and maps outside of this interval are all zero. 
\end{enumerate}
\end{theorem}

Combining Gabriel's theorem and the Krull-Schmidt-Remark theorem implies that if $M$ is a zigzag persistence module with finite support, then there is a finite set of intervals $\{[\alpha_i,\beta_i]\}_i$ such that $$M=\bigoplus\limits_i\mathbb{F}[\alpha_i,\beta_i]$$ [G]. This complete discrete invariant is called $M$'s \textbf{indecomposable decomposition} or \textbf{barcode}.  An example of a zigzag persistence module and its indecomposable decomposition is given in Fig. \ref{fig:Zigzag}.

\begin{figure}[h]
    \centering
    \includegraphics[width=6cm]{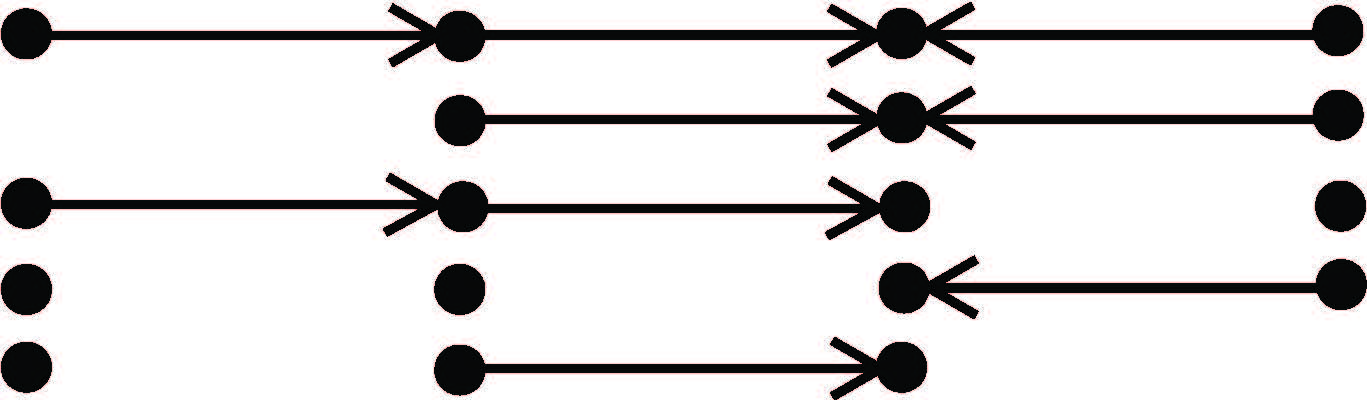}
    \vspace{.2cm}
    \caption{Depicted is a zigzag persistence module $M$. The indecomposable decomposition of $M$ is $M=\mathbb{F}[1,4]\bigoplus\mathbb{F}[2,4]\bigoplus\mathbb{F}[1,3]\bigoplus\mathbb{F}[1,1]\bigoplus\mathbb{F}[2,2]$ $\bigoplus\mathbb{F}[3,4]\bigoplus\mathbb{F}[1,1]\bigoplus\mathbb{F}[2,3]$, where the summands are listed from top to bottom. Notice that $\dim(M_{1})=4=\dim(M_4)$ and $\dim(M_{2})=5=\dim(M_3)$.}
    \label{fig:Zigzag}
\end{figure}

 \subsection{Classification of $n>1$-Parameter Persistence Modules}
When $n>1$, $n$-parameter persistence modules have no such complete discrete classification [CZ]. Based on the notion of cross ratio, Carlsson and Zomorodian found a continuous family of non-isomorphic $2$-parameter persistence modules when $\mathbb{F}=\mathbb{R}$. This family is depicted in Fig. \ref{fig:ContinuousInvariants}, and we will describe it now. Define the linear maps $a,b,c:\mathbb{F}^2\rightarrow\mathbb{F}$ by $a(x,y):=x$, $b(x,y):=y$, and $c(x,y):=x+y$. Let $\lambda_1,\lambda_2\in\mathbb{R}\backslash\{0,1\}$ such that $\lambda_1\neq\lambda_2$. Define $d_i(x,y)=x+\lambda_i y$ and let $M_i$ denote the module illustrated in Fig. \ref{fig:ContinuousInvariants} with $d:=d_i$ for each $i$. 

Suppose there was an isomorphism $\Gamma: M_1\rightarrow M_2$. The fact that $\Gamma$ must commute with the $^{M_1}\phi$ and $^{M_2}\phi$ maps implies each of the following:

\begin{enumerate}
\item{$\Gamma$ will be fully determined by $\Gamma|_{(0,0)}$.}
\item{$\Gamma|_{(i,j)}=\Gamma|_{(0,0)}$ for all $i+j\leq 2$.}
\item{The following kernels must match: $\ker(a\circ\Gamma|_{(0,2)})=\ker(\Gamma|_{(0,3)}\circ a)$. Similar statements can be made for maps $b$ and  $c.$ The commutativity of the $\Gamma$ and $d_i$ maps implies that $\ker(\Gamma|_{(3,0)}\circ d_1)= \ker(d_2\circ \Gamma|_{(2,0)})$.}
\end{enumerate}

\noindent Notice that $\Gamma|_{(i,j)}\in GL_1(\mathbb{R})$ for $i+j=3$. Thus $\Gamma|_{(i,j)}$ is merely a scalar for $i+j=3$. Combining this insight with the second and third statements from the list above implies that $$\ker(f\circ\Gamma|_{(0,0)})=\ker(f)$$ for each map $f\in\{a,b,c\}$. Thus $\Gamma|_{(0,0)}^{-1}$ must preserve each $\ker(f)$. Note that $\ker(a)$ is the $y$-axis, $\ker(b)$ is the $x$-axis, and $\ker(c)$ is the anti-diagonal.  In order to have $\Gamma|_{(0,0)}^{-1}$ preserve these three lines, basic linear algebra implies that $\Gamma|_{(0,0)}=k\cdot\Id_{\mathbb{F}^2}$ for some $k\in\mathbb{R}$. 

Now consider the commutativity of $\Gamma$ and the $d_i$ maps, which yields $\ker(\Gamma|_{(3,0)}\circ d_1)= \ker(d_2\circ \Gamma|_{(2,0)})$. Since $\Gamma|_{(3,0)}=\Gamma|_{(0,0)}$ and $\Gamma|_{(2,0)}$ are scalar maps by above, this implies that $\ker(d_1)=\ker(d_2)$, which is not true. Thus no such isomorphism $\Gamma$ is possible and we have a continuous family of choices (the number $\lambda$ defining map $d$) that yields non-isomorphic multiparameter persistence modules. For any $n>2$, we may embed this family into the $n$-parameter lattice, yielding a continuous family of non-isomorphic $n$-parameter persistence modules.  Any complete invariant of multiparameter persistence modules will thus need a continuous aspect. See [BE] for other infinite families of non-isomorphic $2$- and $3$-parameter persistence modules. The work in [BE] and [EH] characterizes subcategories of $n$-parameter persistence modules which have a complete discrete invariant.
 
\begin{figure}[h]
    \centering
    \includegraphics[scale=.6]{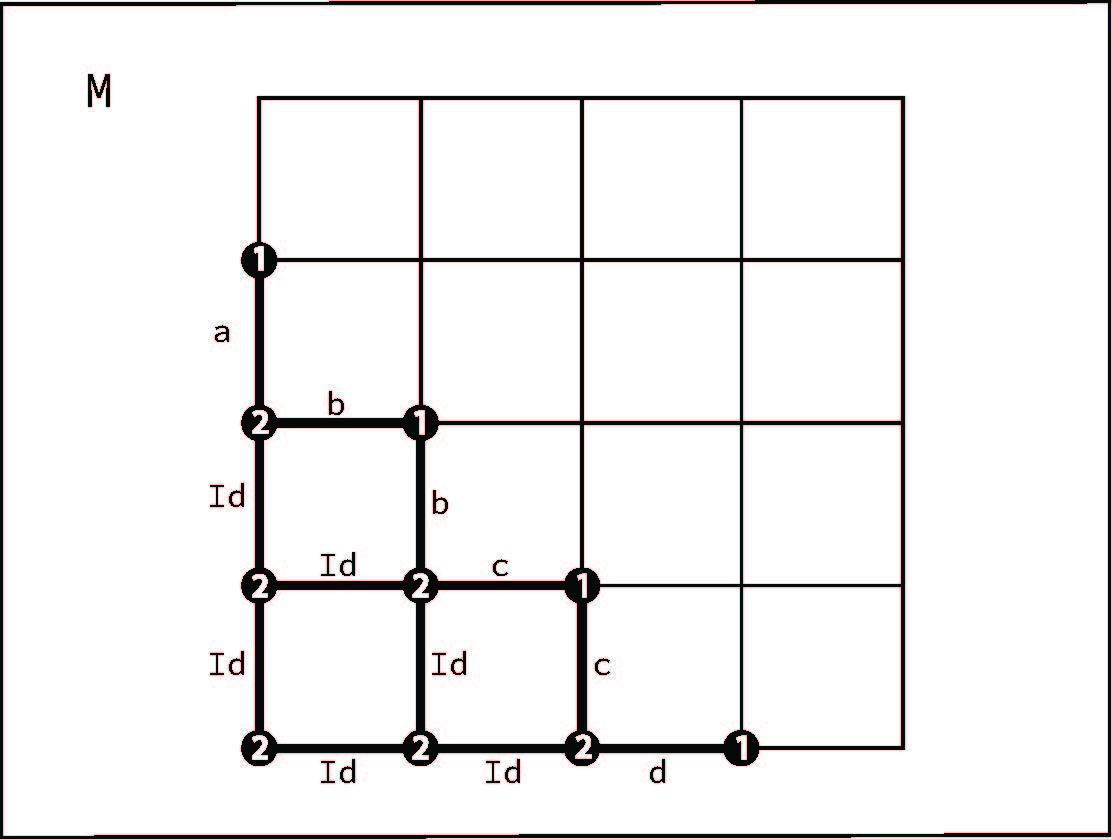}
    \vspace{.2cm}
    \caption{This is (a slight variation of) the continuous family of non-isomorphic $2$-parameter persistence modules explored in [CZ]. The dimensions of each nonzero vector space is shown in white. Thinner edges represent the zero map. We assume that $a(x,y):=x$, $b(x,y):=y$, and $c(x,y):=x+y$ and allow $d$ to vary in order to create non-isomorphic modules.}
    \label{fig:ContinuousInvariants}
\end{figure}

\subsection{The Multigraded Betti Numbers}
Despite their incompleteness, discrete invariants still provide some insight into the properties of multiparameter persistence modules. Such invariants have been explored in various papers, including [CZ, CSk, CL, CSZ, K]. One such invariant is the multigraded Betti numbers $\beta_j^M$, which we now define.

Given a persistence module $M$, $\beta_0^M$ is defined as follows: At each $\alpha\in\mathbb{N}^n$, consider the vector space $$W_\alpha^M:= M_\alpha \Big/ \sum\limits_i im(^M\phi_{\alpha-e_i}^{\alpha}).$$ Then the $0^{th}$ multigraded Betti number of $M$ is the function $\beta_0^M:\mathbb{N}^n\rightarrow \mathbb{N}$  defined by $$\beta_{0}^M(\alpha):=\dim(W_\alpha^M)$$ for all $\alpha\in\mathbb{N}^n$. Define a multiset $\xi_0(M)$ by defining the multiplicity of $\alpha\in\mathbb{N}^n$ in $\xi_0(M)$ to be $\beta_0^M(\alpha)$ [CZ]. Let $F_0(M):=F(\xi_0(M))$ to be the free module associated to $M$. For simplicity, we will often shorten $F_0(M)$ to $F_0$.

Notice that $\xi_0$ (equivalently, $\beta_0^M$) will trivially respect direct sums, namely $\xi_0(M\bigoplus N)=\xi_0(M)+\xi_0(N)$. In the case of free modules $F(S)$, the multiset $\xi_0(F(S))$ will thus be particularly simple; in the base case, $\xi_0(\mathbb{F}_n(\alpha))$ consists of the point $\alpha$ with multiplicity $1$. Because $\xi_0$ respects direct sums, this yields that $\xi_0(F(S))=S$ for every multiset $S$. As such, $\xi_0(F_0(M))=\xi_0(M)$, and thus $W_\alpha^{M}\cong W_\alpha^{F_0(M)}$ for all $\alpha\in\mathbb{N}^n$.

There is be a surjective homomorphism  $\gamma:F_0(M)\rightarrow M$ defined by the following: For each $\alpha\in\mathbb{N}^n$, let $[w_\alpha^1], [w_\alpha^2],..., [w_\alpha^{m_{0,\alpha}}]$ be a basis for $W_\alpha^M$ and $[v_\alpha^1], [v_\alpha^2],..., [v_\alpha^{m_{0,\alpha}}]$ be a basis for $W_\alpha^{F_0(M)}$. It follows that that $\{^M\phi_\beta^\alpha (w_\beta^k) | \beta\leq\alpha, k\leq m_{0,\beta}\}$ forms a basis for $M_\alpha$, and similarly  $\{^{F_0}\phi_\beta^\alpha (v_\beta^k) | \beta\leq\alpha, k\leq m_{0,\beta}\}$ is a basis for $(F_0)_\alpha$. Define $\gamma(v_\alpha^i):=w_\alpha^i$ for all $\alpha, i$. Because $\gamma$ is a homomorphism, by definition it must commute with the $^F\phi$ and $^M\phi$ maps. Thus $$\gamma(w)=\gamma \Big(\sum\limits_{\beta\leq \alpha} \sum\limits_{k\leq m_{0,\beta}} c_\beta^k\ \! ^{F_0}\phi_\beta^\alpha (v_\beta^k) \Big)=\sum\limits_{\beta\leq \alpha} \sum\limits_{k\leq m_{0,\beta}} c_\beta^k\ \! ^M\phi_\beta^\alpha \circ\gamma (v_\beta^k)=\sum\limits_{\beta\leq \alpha} \sum\limits_{k\leq m_{0,\beta}} c_\beta^k\ \! ^M\phi_\beta^\alpha (w_\beta^k).$$ As such, $\gamma(w)$ is defined for all $w\in F_0$.

Notice that $\gamma$ is unique up to composition with an isomorphism of $M$. We denote the kernel of $\gamma$ by $K_0(M)$. This implies that $F_0(M)/K_0(M)\cong M$. Unless $n=1$, $K_0(M)$ may not be a free module, despite being a submodule of $F_0(M)$.  Let $\xi_1(M):=\xi_0(K_0(M))$ and $\beta_1^M:=\beta_0^{K_0}$. Now define $F_1(M)=F_0(K_0(M))=F(\xi_1(M))$ to be the free module associated to $K_0(M)$. As above, we may create a surjection $F_1(M)\twoheadrightarrow K_0(M)$. Denote the kernel of such by $K_1(M)$. Iterate this process, defining new modules $$F_j(M):= F_0(K_j(M)) \text{ and } K_j(M)=K_0(K_{j-1}(M))$$ for all $j$. That is, $F_j(M)$ is the free module associated to $K_{j-1}(M)$ and $K_j(M)$ is the kernel of the surjective homomorphism $F_j(M)\twoheadrightarrow K_{j-1}(M)$. Thus we have maps $$\cdot\cdot\cdot K_2\hookrightarrow F_2 \twoheadrightarrow K_1 \hookrightarrow F_1 \twoheadrightarrow K_0 \hookrightarrow F_0 \twoheadrightarrow M.$$ This gives rise to a minimal length free resolution of $M$,  $\cdot\cdot\cdot F_2 \rightarrow F_1 \rightarrow F_0 \rightarrow M \rightarrow 0$. Hilbert's Syzygy Theorem [H] implies that $F_j=0$ for all $j>n$ whenever $M$ is an $n$-parameter persistence module. For all $j$, define the $j^{th}$ multigraded Betti numbers of $M$ by $$\beta_j^M:=\beta_{0}^{K_{j-1}}=\beta_{0}^{F_{j}}.$$ We again have an equivalent notion, the multiset $$\xi_j(M):=\xi_0(K_{j-1})=\xi_0(F_j)$$ introduced in [CZ]. Let $m_{j,\alpha}:=\beta_j^M(\alpha)$. That is, $m_{j,\alpha}$ denotes the multiplicity of $\alpha$ in $\xi_j(M)$. Notice that Hilbert's Syzygy Theorem implies that $\beta_j^M=0$ for $j>n$. An example of the concepts $F_j(M), K_j(M),$ and $\xi_j(M)$ is shown in Fig. \ref{fig:Example}.

\begin{figure}
    \centering
    \includegraphics[scale=.45]{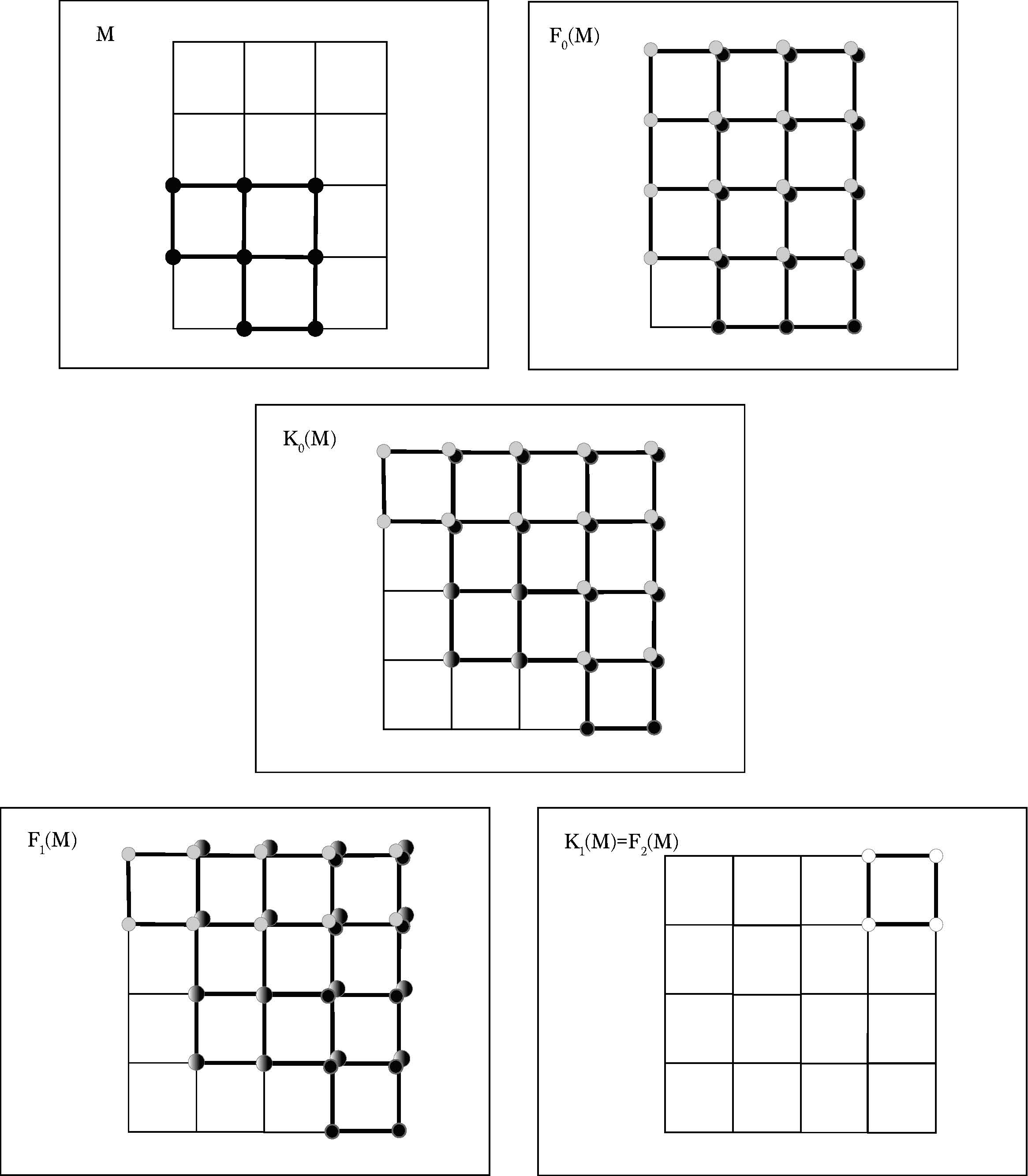}
    \vspace{.2cm}
    \caption{This is an example of a multiparameter persistence module $M$ and its related modules $F_j(M), K_j(M)$ and the related discrete invariants $\xi_j(M)$. In each module, each dot at $\alpha\in\mathbb{N}^2$ represents a basis element of the $\alpha$-vector space. Thick edges represent the identity map on components (i.e. gray basis elements map to gray basis elements, black basis elements map to black basis elements, etc.), while the thinner edges represent the zero map. The coloration and maps of $K_j(M)$ are induced from those of $F_j(M)$. The gradient dots in $(K_0)_\alpha$ represent the vector $(1,-1)\in (F_0)_\alpha$. The white dots in $(K_1)_\alpha$ represent the vector $(1,-1,-1)\in (F_1)_\alpha$ where the ordered basis of $(F_1)_\alpha$ is the black basis vector, the gray basis vector, then the gradient basis vector. Notice that $K_i=F_j= 0$ for $i>1, j>2$. We have $\xi_0(M)=\{(1,0), (0,1)\}$, $\xi_1(M)=\{(3,0), (1,1), (1,3), (0,4)\}$, and $\xi_2(M)=\{(3,3)\}$. For $j>2$, $\xi_j(M)=\emptyset.$ }
    \label{fig:Example}
\end{figure}

\subsection{Properties of $F_j(M)$ and $K_j(M)$}

Fix $j\in[1,n]$. As justified in the previous section, let $[v_\mu^1], [v_\mu^2],..., [v_\mu^{m_{j,\mu}}]$ be a basis for $W_\mu^{F_j(M)}$ and $v_\mu^1, v_\mu^2,..., v_\mu^{m_{j,\mu}}\in F_j(M)_\mu$ be representatives of these classes. Recall that $\{^{F_j}\phi_\mu^\alpha (v_\mu^k) | \mu\leq\alpha, k\leq m_{j,\mu}\}$ is a basis for $(F_j)_\alpha$.  We will abuse notation to write this more concisely, namely writing \begin{equation}F_j(M)_\alpha=\Big\{\!\sum\limits_{\mu\leq\alpha}\sum\limits_{k\leq m_{j,\mu}}\!\!\!c_\mu^k v_\mu^k\ |\ c_\mu^k\in \mathbb{F}\ \forall\ k\Big\}\end{equation} where  $v_\mu^k$ really means $^{F_j}\phi_\mu^\alpha (v_\mu^k)$. By the definition of free multiparameter persistence modules, we have that \begin{equation}\!^{F_j}\phi_{\alpha}^{\beta}\Big(\!\sum\limits_{\mu\leq\alpha}\sum\limits_{k\leq m_{i,\mu}}\!\!\!c_\mu^k v_\mu^k\Big)=\!\sum\limits_{\mu\leq\alpha}\sum\limits_{k\leq m_{j,\mu}}\!\!\!c_\mu^k v_\mu^k\in (F_j)_\beta \end{equation} since each $^{F_j}\phi$ maps act identically on each of $F_j$'s summands. From this observation, we obtain the following three properties of $F_j$ and $K_j\leq F_j$:

\begin{adjustwidth}{30pt}{30pt}

\noindent\textbf{[Property 1]} $^{F_j}\phi_\alpha^\beta$ is injective for all $\alpha\leq\beta$. In particular, $(F_j)_\alpha\cong\ \!^{F_j}\!\phi_{\alpha}^{\beta}((F_j)_\alpha)\leq (F_j)_{\beta}$. Let $(F_j)_\alpha^\beta:=\ ^{F_j}\!\phi_{\alpha}^{\beta}((F_j)_\alpha)$. 

\vspace{.3cm}

\noindent\textbf{[Property 2]}  Because $(K_j)_\alpha\leq (F_j)_\alpha$, Property 1 also implies that $(K_j)_\alpha\cong\ ^{F_j}\phi_\alpha^\beta ((K_j)_\alpha)\leq (F_j)_\beta$. Let $(K_j)_\alpha^\beta:=\ ^{F_j}\phi_\alpha^\beta ((K_j)_\alpha)$. Both of the identifications $(F_j)_\alpha^\beta, (K_j)_\alpha^\beta\leq (F_j)_\beta$ are coodinate-wise identifications by equation (2).

\vspace{.3cm}

\noindent\textbf{[Property 3]} $^{K_j}\phi_\alpha^\beta=\ \! ^{F_j}\phi_\alpha^\beta|_{(K_j)_\alpha}$ is an injective map for all $\alpha\leq\beta$ since $^{F_j}\phi_\alpha^\beta$ is. 
\vspace{.3cm}
\end{adjustwidth}

\noindent These three properties will be used repeatedly throughout the proof of our main result in Section 2.2.

\begin{lemma} Consider any $2$-parameter persistence module $M$. Then $[(K_j)_{\alpha+e_1}^{\alpha+e_1+e_2}\cap (K_j)_{\alpha+e_2}^{\alpha+e_1+e_2}]\leq (F_j)_\alpha^{\alpha+e_1+e_2}$. \end{lemma}
 
\begin{proof} Notice that such a statement makes sense, as $(K_j)_{\alpha+e_1}^{\alpha+e_1+e_2}, (K_j)_{\alpha+e_2}^{\alpha+e_1+e_2}$, and $(F_j)_\alpha^{\alpha+e_1+e_2}$ are each subspaces of $(F_j)_{\alpha+e_1+e_2}$. Let $v\in I_{j,\alpha}:=[(K_j)_{\alpha+e_1}^{\alpha+e_1+e_2}\cap (K_j)_{\alpha+e_2}^{\alpha+e_1+e_2}]$. By eq. (1), $v\in (K_j)_{\alpha+e_i}^{\alpha+e_1+e_2}\subset (F_j)_{\alpha+e_i}^{\alpha+e_1+e_2}$ implies that $v$ can be written as \begin{equation}v=\sum\limits_{\mu\leq\alpha+e_i}\sum\limits_{k\leq m_{j,\mu}}\!\!\!c_\mu^k v_\mu^k\in (F_j)_{\alpha+e_1+e_2}\end{equation} for some coefficients $c_\mu^k \in \mathbb{F}$. Furthermore, because this must be true for $i=1,2$, by uniqueness it must be the case that $c_\mu^k=0$ whenever $\mu\nleq\alpha+e_i$ for $i\in 1,2$. Thus Eq. (3) reduces to $$v=\sum\limits_{\mu\leq\alpha}\sum\limits_{k\leq m_{j,\mu}}\!\!\!c_{\mu}^k v_\mu^k,$$ which is an element of $(F_j)_\alpha^{\alpha+e_1+e_2}$ by eq. (2). \end{proof}

\section{A Combinatorial Formula for the Bigraded Betti Numbers}

We defined $n$-parameter persistence modules using the lattice $\mathbb{N}^n$, however we can trivially extend to the $\mathbb{Z}^n$ lattice. Namely, define $M_\beta=0$ for all $\beta\in\mathbb{Z}^n\backslash\mathbb{N}^n$ and let $^M\phi_\alpha^\beta$ be the zero map whenever $\alpha\in\mathbb{Z}^n\backslash\mathbb{N}^n$. This will not impact any of the properties of $M$ that we care about (such as $\xi_j(M)$), and will simplify the explanations needed throughout the next sections.

For $\alpha\in\mathbb{N}^n$, let the \textbf{into-$\alpha$} \textbf{frame} be the restriction of $M$ to the path in the $\mathbb{Z}^2$ lattice whose vertices are $\{\alpha, \alpha-e_1, \alpha-e_2, \cdot\cdot\cdot, \alpha- e_n\}$ and edges are $\{\alpha-e_i\rightarrow\alpha\}_i$. Similarly, let the \textbf{$\alpha$-outward frame} be the restriction of $M$ to the path with vertices $\{\alpha, \alpha+e_1, \alpha+e_2, \cdot\cdot\cdot, \alpha+ e_n\}$ and edges $\{\alpha\rightarrow\alpha+e_i\}_i$. 



\begin{theorem} Let $M$ be a finitely presented $2$-parameter persistence module. For each $\alpha\in\mathbb{N}^n$, let $z_\alpha^M$ denote the multiplicity of $\mathbb{F}[\alpha,\alpha]$ in the barcode of the $\alpha$-outward frame of $M$ and let $y_\alpha^M$ denote the multiplicity of $\mathbb{F}[\alpha,\alpha]$ in the barcode of the into-$\alpha$ frame of $M$. Then for all $\alpha\in\mathbb{N}^2$, 

\begin{equation}\beta_{j}^M(\alpha)= \begin{cases} 
      y_\alpha^M & j=0 \\
       y_{\alpha}^M-\dim(M_{\alpha})+\dim(M_{\alpha-e_1})+\dim(M_{\alpha-e_2})-\dim(M_{\alpha-e_1-e_2})+z^M_{\alpha-e_1-e_2} & j=1 \\
      z_{\alpha-e_1-e_2}^M & j=2.
   \end{cases}\end{equation}
\end{theorem}

\begin{remark} Theorem 2.1 is known and has previously been proven using techniques from commutative algebra; however, many researchers in persistent homology are not well-versed in commutative algebra. Our proof may thus be more accessible to these researchers, as it will require only basic linear algebra. A summary of the commutative algebra proof is as follows: consider $M$ as an $\mathbb{N}^2$ graded module over $\mathbb{F}[x_1,x_2]$. The bigraded Betti numbers may be equivalently defined using tensor products of $\mathbb{N}^2$ graded modules over $\mathbb{F}[x_1,x_2]$. The formula in Theorem 2.1 is found by tensoring $M$ with the Koszul complex on $x_1$ and $x_2$. See [P] for further details.
\end{remark}

 \begin{proof}
 When $j=0$, the result is trivial by the definition of $\beta_0^M$. Now we wish to determine $\beta_1^M$, or equivalently $\xi_1(M)$. 
Let $\alpha+e_1+e_2\in\mathbb{N}^2$. By definition, we have
\begin{equation}
\begin{aligned}
m_{1,\alpha+e_1+e_2}&=\dim((K_0)_{\alpha+e_1+e_2})-\dim((K_0)_{\alpha+e_1}^{\alpha+e_1+e_2}+ (K_0)_{\alpha+e_2}^{\alpha+e_1+e_2})\\
 &=\dim((K_0)_{\alpha+e_1+e_2})-\dim((K_0)_{\alpha+e_1}^{\alpha+e_1+e_2})-\dim((K_0)_{\alpha+e_2}^{\alpha+e_1+e_2})+\dim(I_{0,\alpha})\\
\end{aligned}
\end{equation}
where $I_{0,\alpha}=[(K_0)_{\alpha+e_1}^{\alpha+e_1+e_2}\cap (K_0)_{\alpha+e_2}^{\alpha+e_1+e_2}]$. Recall that $(K_0)_{\alpha+e_i}^{\alpha+e_1+e_2}\cong (K_0)_{\alpha+e_i}$ by Property 2. Thus the dimension of $(K_0)_{\alpha+e_i}^{\alpha+e_1+e_2}$ is known. As such, we simply need to know the dimension of $I_{0,\alpha}$.

Recall that $I_{0,\alpha} \leq (F_0)_\alpha^{\alpha+e_1+e_2}$ by lemma 1.1. In fact, $$(K_0)_\alpha^{\alpha+e_1+e_2}\leq I_{0,\alpha} \leq (F_0)_\alpha^{\alpha+e_1+e_2}$$ as follows: Suppose $v\in (K_0)_\alpha^{\alpha+e_1+e_2}$. Then there exists $w\in (K_0)_\alpha$ such that $v=\ \! ^{K_0}\phi_{\alpha}^{\alpha+e_1+e_2}(w)={^{K_0}\phi_{\alpha}^{\alpha+e_i}\circ\!^{K_0}\phi_{\alpha+e_i}^{\alpha+e_1+e_2}}(w)$, where the first equality comes from the definition of $K_\alpha^{\alpha+e_1+e_2}$, and the second from the required commutativity of the maps within a multiparameter persistence module. Thus $v\in im(^{K_0}\phi_{\alpha+e_i}^{\alpha+e_1+e_2})=(K_0)_{\alpha+e_i}^{\alpha+e_1+e_2}$ for each $i$.

Because $(F_0)_\alpha^{\alpha+e_1+e_2}\cong (F_0)_\alpha$ (Property 1), we can identify $I_{0,\alpha}$ with a subspace $I_{0,\alpha}^\alpha\leq (F_0)_\alpha$. More specifically, tracing through the identifications yields that \begin{equation}I_{0,\alpha}^\alpha=\{v|\ \! ^{F_0}\phi_\alpha^{\alpha+e_i}(v)\in (K_0)_{\alpha+e_i}\ \forall\ i\}\leq (F_0)_\alpha. \end{equation} The above identification $(K_0)_\alpha^{\alpha+e_1+e_2}\leq I_{0,\alpha} \leq (F_0)_\alpha^{\alpha+e_1+e_2}$ can then be rewritten as the isomorphic identifications $$(K_0)_\alpha^{\alpha}\leq I_{0,\alpha}^\alpha \leq (F_0)_\alpha.$$  Now consider $I_{0,\alpha}^\alpha/ (K_0)_\alpha^\alpha\cong I_{0,\alpha} / (K_0)_\alpha^{\alpha+e_1+e_2}$. We claim that $I_{0,\alpha}^\alpha/ (K_0)_\alpha^\alpha\cong\{w\in M_\alpha | ^M\phi_{\alpha}^{\alpha+e_i}(w)=0\ \forall\ i\}$ via the isomorphism $\Gamma:F_0/K_0\rightarrow M$. Notice that $\Gamma$ is the unique homomorphism such that $\Gamma \circ q=\gamma$, where $q:F_0\rightarrow F_0/K_0$ is the quotient map and $\gamma$ is the surjection $F_0\twoheadrightarrow M$. 

We first show that $\Gamma(I_{0,\alpha}^\alpha / (K_0)_\alpha^\alpha)\subseteq \{w\in M_\alpha | ^M\phi_{\alpha}^{\alpha+e_i}(w)=0\ \forall\ i\}$. Let $v$ be a representative of some nonzero class $[v]\in I_{0,\alpha}^\alpha/ (K_0)_\alpha^\alpha\subseteq (F_0)_\alpha / (K_0)_\alpha^\alpha$. Then $\Gamma ([v])=w\neq 0$ since $\Gamma$ is an isomorphism. Furthermore, Eq. (6) implies that $$\gamma_{\alpha+e_i}\circ\ \! ^{F_0}\phi_{\alpha}^{\alpha+e_i}(v)=0.$$ Because the $\gamma, ^{F_0}\phi,$ and $^M\phi$ maps must commute, this implies that $$^M\phi_{\alpha}^{\alpha+e_i}(w)=\ \! ^M\phi_{\alpha}^{\alpha+e_i}\circ \gamma_\alpha (v)=\gamma_{\alpha+e_i}\circ\ \! ^{F_0}\phi_{\alpha}^{\alpha+e_i}(v)=0$$ for $i=1,2$. Thus $\Gamma(I_{0,\alpha}^\alpha)\subseteq \{w\in M_\alpha | ^M\phi_{\alpha}^{\alpha+e_i}(w)=0\ \forall\ i\}$.

Now we show that $I_{0,\alpha}^\alpha / (K_0)_\alpha^\alpha\subseteq \Gamma^{-1}(\{w\in M_\alpha | ^M\phi_{\alpha}^{\alpha+e_i}(w)=0\ \forall\ i\})$. Let $0\neq w\in M_\alpha$ such that $^M\phi_{\alpha}^{\alpha+e_i}(w)=0$ for $i=1,2$. Let $$\Gamma^{-1}(w)=:[v]\in (F_0)_\alpha / (K_0)_\alpha^\alpha,$$ and let $v\in (F_0)_\alpha$ be a representative of the class $[v]$. By commutativity of the $\gamma, ^{F_0}\phi,$ and $^M\phi$ maps, we have $$\gamma_{\alpha+e_i}\circ\ \! ^{F_0}\phi_{\alpha}^{\alpha+e_i}(v)=\ \! ^M\phi_{\alpha}^{\alpha+e_i}\circ\gamma_{\alpha}(v)=\ \! ^M\phi_{\alpha}^{\alpha+e_i}(w)=0.$$ In particular, $^{F_0}\phi_{\alpha}^{\alpha+e_i}(v)\in (K_0)_{\alpha+e_i}$ for each $i$, implying  $v\in I_{0,\alpha}^\alpha$ by Eq. (6). Thus $[v]\in I_{0,\alpha}^\alpha / (K_0)_\alpha^\alpha$, as desired. 

As such, we have shown that $$I_{0,\alpha} / (K_0)_\alpha\cong I_{0,\alpha}^\alpha/ (K_0)_\alpha^\alpha\cong\{w\in M_\alpha | ^M\phi_{\alpha}^{\alpha+e_i}(w)=0\ \forall\ i\}.$$ Thus $$\dim(I_{0,\alpha} / (K_0)_{\alpha}^{\alpha+e_1+e_2})=\dim(I_{0,\alpha}^\alpha/ (K_0)_\alpha^\alpha)=\dim(\{w\in M_\alpha | ^M\phi_{\alpha}^{\alpha+e_i}(w)=0\ \forall\ i\}),$$ which trivially equals $z^M_\alpha$. Then we may write $\dim(I_{0,\alpha})=\dim((K_0)_\alpha)+z^M_\alpha$. Combining with equations (4) and (5) yields that the multiplicity of $\alpha+e_1+e_2$ in $\xi_0(K_0)$ is 
\begin{equation}\begin{aligned} m_{1,\alpha+e_1+e_2}&=\dim((K_0)_{\alpha+e_1+e_2})-\dim((K_0)_{\alpha+e_1})-\dim((K_0)_{\alpha+e_2})+\dim((K_0)_\alpha)+z^M_\alpha\\
& =m_{0,\alpha+e_1+e_2}-\dim(M_{\alpha+e_1+e_2})+\dim(M_{\alpha+e_1})+\dim(M_{\alpha+e_2})-\dim(M_\alpha)+z^M_\alpha, \end{aligned}\end{equation} proving Theorem 2.1 for the $j=1$ case.

To find $\xi_2(M)=\xi_1(K_0)$, we will apply equation (7) to $K_0$. Notice that the $\alpha$-outward frames of $K_0$ are fully determined (up to isomorphism) by the dimension vector of $K_0$ since each map $^{K_0}\phi_{\alpha}^{\alpha+e_i}$ is injective. As such, $z_\alpha^{K_0}=0$ for all $\alpha\in\mathbb{N}^n$. Rewriting equation (7) in terms of $K_0$ (instead of $M$) yields that, for $\alpha\in\mathbb{N}^2$,

\begin{equation*}\begin{aligned} m_{2,\alpha+e_1+e_2}&=m_{1,\alpha+e_1+e_2}-\dim((K_0)_{\alpha+e_1+e_2})+\dim((K_0)_{\alpha+e_1})+\dim((K_0)_{\alpha+e_2})-\dim((K_0)_\alpha) \\
& =z^M_\alpha, \end{aligned}\end{equation*}
where the second equality follows from replacing $m_{1,\alpha+e_1+e_2}$ by equation (7).

\end{proof}

\begin{remark} Note that Theorem 2.1 implies that the into-$\alpha$ and $\alpha$-outward frames determine the bigraded Betti numbers of $M$. The converse of this is not true; For example, consider the modules in Fig. \ref{fig:Counterexample}.

\begin{figure}[h]
    \centering
    \includegraphics[scale=.5]{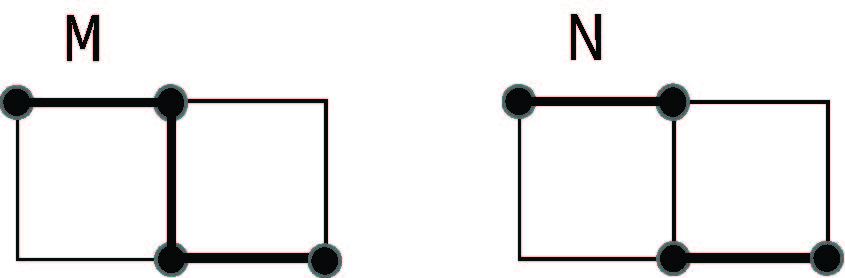}
    \vspace{.2cm}
    \caption{Modules $M$ and $N$ both have $\xi_0=\{(1,0), (0,1)\}$ and $\xi_1=\{(0,2), (1,1), (2,1), (3,0)\}$. However, these modules clearly have distinct into-$(0,1)$ frames. Thus, though the into-$\alpha$ and $\alpha$-outward frames determine the multigraded Betti numbers, the converse is not true.}
    \label{fig:Counterexample}
\end{figure}
\end{remark}

\begin{remark}
   The formulas in Theorem 2.1 indicate a close relationship between the dimension vector of $M$ and the multisets $\xi_j(M)$ when $M$ is a $2$-parameter persistence module. This relationship can be generalized to $n$-parameter persistence modules as follows: by the definitions of the modules $F_j$ and $K_j$, we have
$\dim(M_\alpha)+\dim((K_0)_\alpha)=\dim((F_0)_\alpha)$ and $\dim((K_j)_\alpha)+\dim((K_{j-1})_\alpha)=\dim((F_j)_\alpha)$ for all $j, \alpha$. Rearranging these allows us to write \begin{equation}\dim(M_\alpha)=\sum\limits_{j}\dim((F_{2j})_\alpha)-\sum\limits_{j}\dim((F_{2j+1})_\alpha)= \sum\limits_{\mu\leq\alpha} \sum\limits_j m_{2j,\mu}-\sum\limits_{\mu\leq\alpha} \sum\limits_j m_{2j+1,\mu}\end{equation} for all $\alpha$, where the second inequality comes from applying equation (1). Equation (8) implies that

\begin{equation*}
\dim(M_{\alpha})-\sum\limits_i \dim(M_{\alpha-e_i})+\sum\limits_{i< j} \dim(M_{\alpha-e_i-e_j})+\cdot\cdot\cdot \pm \dim(M_{\alpha-e_1-e_2-\cdot\cdot\cdot -e_n})=\sum\limits_j m_{2j,\alpha}-\sum\limits_j m_{2j+1,\alpha}.
\end{equation*} This result is well known and was first introduced by Hilbert [H].

\end{remark}

\end{document}